\documentclass[reqno]{amsart}
\usepackage{amssymb, stmaryrd}
\usepackage{mathrsfs}
\usepackage{amscd}
\usepackage{pdfsync} 

\newtheorem{theorem}{Theorem}[section]
\newtheorem{lemma}[theorem]{Lemma}

\newtheorem{definition}[theorem]{Definition}
\newtheorem{example}[theorem]{Example}

\newtheorem{proposition}[theorem]{Proposition}
\newtheorem{corollary}[theorem]{Corollary}

\newcommand{\F}{\boldsymbol{F}}
\newcommand{\FF}{\overline{\boldsymbol{F}}}

\usepackage{color}  
\definecolor{darkgreen}{rgb}{0.03, 0.5, 0.03}
 \newcommand{\bc}{\color{blue}} %

 \newcommand{\Na} {\mbox{\rm Na}}

  \newcommand{\Max} {\mbox{\rm Max}}
    \newcommand{\Spec} {\mbox{\rm Spec}}
      \newcommand{\QMax} {\mbox{\rm QMax}}
        \newcommand{\QSpec} {\mbox{\rm QSpec}}

          \newcommand{\st}{\star}

\newcommand{\na}{{\mbox{\rm Na}}(D,\st)}

\newcommand{\cal} {\mathcal}

\newcommand{\Fb}{\boldsymbol{\overline{F}}}
\newcommand{\f}{\boldsymbol{{f}}}

  \newcommand{\stf} {\star{_{\!{_f}}}}
    
    \newcommand{\stt} {\widetilde{\star}}

\begin{document}


\title{Idempotence and divisorialty in Pr\"ufer-like domains}

\author{Marco Fontana}
\address{Dipartimento di Matematica e Fisica, Universit\`a degli Studi Roma Tre, Largo San Leonardo Murialdo, 1, 00146 Roma, Italy}
\email{fontana@mat.uniroma3.it}

\author{Evan Houston}
\address{Department of Mathematics and Statistics, University of North Carolina at Charlotte,
Charlotte, NC 28223 U.S.A.}
\email{eghousto@uncc.edu}

\author{Mi Hee Park}
\address{Department of Mathematics,
Chung-Ang University, Seoul  06974, Korea}
\email{mhpark@cau.ac.kr}

\thanks{ The first-named author was partially supported   by {\sl GNSAGA} of {\sl Istituto Nazionale di Alta Matematica}.}

\thanks{{The second-named author was supported by a grant from the Simons Foundation (\#354565).}}

\thanks{The third-named author was supported by the National Research Foundation of Korea (NRF) grant funded by the Korea government(MSIP) (No. 2015R1C1A2A01055124).}

{\bc  \date{\today}}
%

%
%
\begin{abstract} Let $D$ be a Pr\"ufer $\st$-multiplication domain, where $\st$ is a semistar operation on $D$.  We show that certain ideal-theoretic properties related to idempotence and divisoriality hold in Pr\"ufer domains, and we use the associated semistar Nagata ring of $D$ to show that the natural counterparts of these properties also hold in $D$.
 \end{abstract}

\maketitle


\section{introduction and preliminaries} Throughout this work, $D$ will denote an integral domain, and $K$ will denote its quotient field.  
Recall that Arnold \cite{ar} proved that $D$ is a Pr\"ufer domain if and only if its associated Nagata ring $D[X]_N$, where $N$ is the set of polynomials in $D[X]$ whose coefficients generate the unit ideal, is a Pr\"ufer domain.  This was generalized to Pr\"ufer $v$-multiplication domains (P$v$MDs) by Zafrullah \cite{z} and Kang \cite{K-89} and to Pr\"ufer $\st$-multiplication domains (P$\st$MDs) by Fontana, Jara, and Santos \cite{FJS-03}.

Our goal in this paper is to show that certain ideal-theoretic properties that hold in Pr\"ufer domains 
 transfer in a natural way to P$\st$MDs. 
For example, we show  that an ideal $I$ of a Pr\"ufer domain is idempotent if and only if it is a radical ideal each of whose minimal primes is idempotent (Theorem~\ref{t:idemp}), and we use a Nagata ring transfer ``machine'' to transfer a natural counterpart of this characterization to P$\st$MDs. 
 For another example, in Theorem~\ref{t:idemp-div} we show that an ideal in a Pr\"ufer domain of finite character is idempotent if and only it is a product of idempotent prime ideals and, perhaps more interestingly, we characterize ideals that are simultaneously idempotent and divisorial as (unique) products of incomparable divisorial idempotent primes; and we then extend this to P$\st$MDs.

Let us review terminology and notation. 
Denote by $\boldsymbol{\overline{F}}(D)$ the set of all nonzero
$D$--submodules of   $K$, and by $\boldsymbol{F}(D)$ the set of
all nonzero fractional ideals of $D$, i.e., $E \in
\boldsymbol{F}(D)$ if $E \in \boldsymbol{\overline{F}}(D)$ and
there exists a nonzero $d \in D$ with $dE \subseteq D$. Let
$\f(D)$ be the set of all nonzero finitely generated
$D$--submodules of $K$. Then, obviously, $\f(D)
\subseteq \boldsymbol{F}(D) \subseteq
\boldsymbol{\overline{F}}(D)$.

Following Okabe-Matsuda \cite{o-m}, a \emph{semistar operation} on $D$ is a map $\star:
\boldsymbol{\overline{F}}(D) \to \boldsymbol{\overline{F}}(D),\ E \mapsto E^\star$,  such that, for all $x \in K$, $x \neq 0$, and
for all $E,F \in \boldsymbol{\overline{F}}(D)$, the following
properties hold:
\begin{enumerate}
\item[$(\star_1)$] $(xE)^\star=xE^\star$;
 \item[$(\star_2)$] $E
\subseteq F$ implies $E^\star \subseteq F^\star$;
\item[$(\star_3)$] $E \subseteq E^\star$ and $E^{\star \star} :=
\left(E^\star \right)^\star=E^\star$.
\end{enumerate}

Of course, semistar operations are natural generalizations of star operations--see the discussion following Corollary~\ref{c:sharp} below.

The semistar operation $\st$ is said to have \emph{finite type} if $E^{\st}=\bigcup \{F^{\st} \mid F\in \f(D), F\subseteq E\}$ for each $E \in \Fb(D)$. To any semistar operation $\st$ we can associate a finite-type semistar operation $\stf$ given by $$E^{\stf}=\bigcup \{F^{\st} \mid F \in \f(D), F \subseteq E\}.$$


We say that a nonzero ideal $I$ of $D$ is a
\emph{quasi-$\star$-ideal} if $I=I^\star \cap D$, a
\emph{quasi-$\star$-prime  ideal} if it is a prime quasi-$\star$-ideal,
and a \emph{quasi-$\star$-maximal  ideal} if it is maximal in the set of
all   proper   quasi-$\star$-ideals. A quasi-$\star$-maximal ideal is  a
prime ideal. We
will denote by $\QMax^{\star}(D)$  ($\QSpec^\star(D)$) the set of all
quasi-$\star$-maximal ideals  (quasi-$\star$-prime ideals) of $D$.
While quasi-$\st$-maximal ideals may not exist, quasi-$\stf$-maximal ideals are plentiful in the sense that each  proper   quasi-$\star_{_{\!
f}}$-ideal is contained in a quasi-$\star_{_{\! f}}$-maximal
ideal.  (See \cite{FL3} for details.) Now we can associate to $\st$ yet another semistar operation: for $E \in \Fb(D)$, set  $$E^{\stt}=\bigcap \{ED_Q\mid Q \in \QMax^{\stf}(D)\}.$$ Then $\stt$ is also a finite-type semistar operation, and we have $I^{\stt} \subseteq I^{\stf} \subseteq I^{\st}$ for all $I \in \Fb(D)$.


%
%

Let $\st$ be a semistar operation on $D$.  Set $N(\st)=\{g \in D[X] \mid c(g)^{\st}=D^{\st}\}$, where $c(g)$ is the \emph{content} of the polynomial $g$, i.e., the ideal of $D$ generated by the coefficients of $g$.  Then $N(\st)$ is a saturated multiplicatively closed subset of $D[X]$, and we call the ring $\na:=D[X]_{N(\st)}$ the \emph{semistar Nagata ring of $D$ with respect to $\st$}.  The domain $D$ is called a \emph{Pr\"ufer $\st$-multiplication domain} (P$\st$MD) if $(FF^{-1})^{\stf}=D^{\stf} \ (=D^{\st})$ for each $F \in \f(D)$ (i.e., each such $F$ is $\stf$-invertible).   (Recall that $F^{-1}=(D:F)=\{u \in K \mid uF \subseteq D\}$.)

In the following two lemmas, we assemble the facts we need about Nagata rings and P$\st$MDs.  Most of the proofs can be found in \cite{FH}, \cite{FL3}, or \cite{fhp}.

\begin{lemma} \label{l:nagata} Let $\st$ be a semistar operation on $D$. Then: \begin{enumerate}
\item $D^{\st}=D^{\stf}$.
\item $\QMax^{\stf}(D)=\QMax^{\stt}(D)$.
\item The map $\QMax^{\stf}(D) \to \Max(\na)$, $P \mapsto P\na$, is a bijection with inverse map $M \mapsto M\cap D$.
\item $P \mapsto P\na$ defines an injective map $\QSpec^{\stt}(D) \to \Spec(\na)$.
\item $N(\st)=N(\stf)=N(\stt)$ and $($hence$)$ $\na={\rm Na}(D,\stf)={\rm Na}(D,\stt)$.
\item For each $E\in \Fb(D)$, $E^{\stt}=E\na \cap K$, and $E^{\stt}\na=E\na$.
\item A nonzero ideal $I$ of $D$ is a quasi-$\stt$-ideal if and only if $I=I\na \cap D$. 
\end{enumerate}
\end{lemma}

\begin{lemma} \label{l:p*md} Let $\st$ be a semistar operation on $D$.  \begin{enumerate}
\item The following statements are equivalent. \begin{enumerate}
\item $D$ is a P$\st$MD.
\item $\na$ is a Pr\"ufer domain.
\item The ideals of $\na$ are extended from ideals of $D$.
\item $D_P$ is a valuation domain for each $P\in \QMax^{\stf}(D)$.
\end{enumerate}
\item Assume that $D$ is a P$\st$MD. Then: \begin{enumerate}
\item $\stt=\stf$ and $($hence$)$ $D^{\st}=D^{\stt}$.
\item The map $\QSpec^{\stf}(D) \to \Spec(\na)$, $P \mapsto P\na$, is a bijection with inverse map $Q \mapsto Q\cap D$.
\item Finitely generated ideals of $\na$ are extended from finitely generated ideals of $D$.
\end{enumerate}
\end{enumerate}
\end{lemma}


\section{Idempotence}

We begin with our basic definition.

\begin{definition} \label{d:idemp} Let $\st$ be a semistar operation on $D$.  An element $E \in \Fb(D)$ is said to be $\st$-idempotent if $E^{\st}=(E^2)^{\st}$.
\end{definition}

Our primary interest will be in (nonzero) $\st$-idempotent \emph{ideals} of $D$.  Let $\st$ be a semistar operation on $D$, and let $I$ be a nonzero ideal of $D$.  It is well known that $I^{\st}\cap D$ is a quasi-$\st$-ideal of $D$.  (This is easy to see: we have $$(I^{\st}\cap D)^{\st} \subseteq I^{\st \st}=I^{\st}=(I\cap D)^{\st}\subseteq (I^{\st}\cap D)^{\st},$$ and hence $I^{\st}=(I^{\st}\cap D)^{\st}$; it follows that $I^{\st}\cap D=(I^{\st}\cap D)^{\st}\cap D$.)  It therefore seems natural to call $I^{\st}\cap D$ the \emph{quasi-$\st$-closure} of $I$.  If we also call $I$ \emph{$\st$-proper} when $I^{\st}\subsetneq D^{\st}$, then it is easy to see that $I$ is $\st$-proper if and only if its quasi-$\st$-closure is a proper quasi-$\st$-ideal. Now suppose that $I$ is $\st$-idempotent. Then $$(I^{\st}\cap D)^{\st}=I^{\st}=(I^2)^{\st}=((I^{\st})^2)^{\st}=(((I^{\st}\cap D)^{\st})^2)^{\st}=((I^{\st}\cap D)^2)^{\st},$$ whence $I^{\st}\cap D$ is a $\st$-idempotent quasi-$\st$-ideal of $D$. A similar argument gives the converse.  Thus a ($\st$-proper) nonzero ideal is $\st$-idempotent if and only if its quasi-$\st$-closure is a (proper) $\st$-idempotent quasi-$\st$-ideal.

Our study of idempotence in Pr\"ufer domains and P$\st$MDs involves the notions of sharpness and branchedness. We recall some notation and terminology.

Given   an integral domain $D$ and a prime ideal  $P \in \Spec(D)$, set
$$
\begin{array}{rl}
\nabla(P) :=& \hskip -8pt  \{ M \in \Max(D) \mid M \nsupseteq P\}\,{\rm and}\\
\Theta(P) :=& \hskip -7pt  \bigcap \{ D_M \mid M \in \nabla(P) \}\,.
\end{array}
$$

We say that $P$  is {\it sharp}  
if $\Theta(P) \nsubseteq D_P$  (see \cite[Lemma 1]{G-66} and \cite[Section 1 and Proposition 2.2]{FHL-10}). The domain $D$ itself is \emph{sharp} (\emph{doublesharp}) if every maximal (prime) ideal of $D$ is sharp.
 (Note that a Pr\"{u}fer domain $D$ is doublesharp if and only if each overring of $D$ is sharp \cite[Theorem 4.1.7]{FHP97}.)
 Now let $\st$ be a semistar operation on $D$.  
 Given a prime ideal  $P \in \QSpec^{\stf}(D)$, set
$$
\begin{array}{rl}
\nabla^{\stf}(P) :=& \{ M \in \QMax^{\stf}(D) \mid M \nsupseteq P\}\, {\rm and}\\
\Theta^{\stf}(P) :=&  \bigcap \{ D_M \mid M \in \nabla^{\stf}(P) \}\,.
\end{array}
$$

Call $P$ {\it $\stf$-sharp} 
if $\Theta^{\stf}(P) \nsubseteq D_P$.
For example, if $\star =d$  is the identity,  then the $\stf$-sharp property  coincides with the sharp property. We then say that $D$ is $\stf$-(double)sharp if  each quasi-$\stf$-maximal (quasi-$\stf$-prime) ideal of $D$ is $\stf$-sharp. (For more on sharpness, see \cite{G-64}, \cite{G-66}, \cite{GH-67}, \cite[page 62]{FHP97}, \cite{FHL-10}, \cite[Chapter 2, Section 3]{FHL-13}, and \cite{fhp}.)

Recall that a prime ideal $P$ of a ring is said to be \emph{branched} if there is a $P$-primary ideal distinct from $P$.  Also, recall that the domain $D$ has \emph{finite character} if each nonzero ideal of $D$ is contained in only finitely many maximal ideals of $D$.


We now prove a lemma that discusses the transfer of ideal-theoretic properties between $D$ (on which a semistar operation $\st$ has been defined) and its associated Nagata ring.

\begin{lemma} \label{l:branched-idempotent}
Let $\st$ be a semistar operation on $D$.
\begin{enumerate}
\item  Let $E \in \Fb(D)$.  Then $E$ is $\stt$-idempotent if and only if $E{\rm Na}(D,\st)$ is idempotent.  In particular, if $D$ is a P$\st$MD, then $E$ is $\stf$-idempotent if and only if $E{\rm Na}(D,\st)$ is idempotent.
\item Let $P$ be a quasi-$\stt$-prime of $D$ and $I$ a nonzero ideal of $D$. Then: \begin{enumerate}
\item  $I$ is $P$-primary in $D$ if and only if $I$ is a quasi-$\stt$-ideal of $D$ and $I\na$ is $P\na$-primary in $\na$.
\item $P$ is branched in $D$ if and only if $P\Na(D,\star)$ is branched in $\Na(D,\star)$.
\end{enumerate}
\item  $D$ has $\stf$-finite character $($i.e., each nonzero element of $D$ belongs to only finitely many $($possibly zero$)$ $M \in \QMax^{\stf}(D))$  if and only if $\Na(D,\star)$ has finite character.
\item  Let $I$ be a quasi-$\stt$-ideal of $D$.  Then $I$ is a radical ideal if and only if $I\Na(D,\star)$ is a radical ideal of $\Na(D,\star)$.
\item Assume that $D$ is a P$\st$MD.  Then: \begin{enumerate}
\item If $P \in \QSpec^{\stf}(D)$, then $P$ is $\stf$-sharp if and only if $P{\rm Na}(D,\st)$ is sharp  in ${\rm Na}(D,\st)$.
\item $D$ is $\stf$-$($double$)$sharp if and only if $\na$ is $($double$)$sharp.
\end{enumerate}
\end{enumerate}
\end{lemma}
\begin{proof}
(1) We use Lemma~\ref{l:nagata}(6). If $E\na$ is idempotent, then $E^{\stt}=E\na \cap K=E^2\na \cap K=(E^2)^{\stt}$.  Conversely, if $E$ is $\stt$-idempotent, then $(E\na)^2=E^2\na=(E^2)^{\stt}\na=E^{\stt}\na=E\na$. The ``in particular'' statement follows because $\stf=\stt$ in a P$\st$MD (Lemma~\ref{l:p*md}(2a)).

(2) (a) Suppose that $I$ is $P$-primary. Then $ID[X]$ is $PD[X]$-primary.  Since $P$ is a quasi-$\stt$-prime of $D$, $P\na$ is a prime ideal of $\na$ (Lemma~\ref{l:nagata}(4)), and then, since $\na$ is a quotient ring of $D[X]$, $I\na$ is $P\na$-primary in $\na$.  Also, again using the fact that $ID[X]$ is $PD[X]$-primary (along with Lemma~\ref{l:nagata}(6)), we have $$I^{\stt} \cap D=I\na \cap D \subseteq ID[X]_{PD[X]} \cap D[X] \cap D =ID[X] \cap D=I,$$ whence $I$ is a quasi-$\stt$-ideal of $D$.  Conversely, assume that $I$ is a quasi-$\stt$-ideal of $D$ and that $I\na$ is $P\na$-primary. Then for $a\in P$, there is a positive integer $n$ for which $a^n \in I\na \cap D=I^{\stt} \cap D=I$. Hence $P={\rm rad}(I)$.  It now follows easily that $I$ is $P$-primary. 

(b) Suppose that $P$ is branched in $D$. Then there is a $P$-primary ideal $I$ of $D$ distinct from $P$, and $I\na$ is $P\na$-primary by (a).  Also by (a), $I$ is a quasi-$\stt$-ideal, from which it follows that $I\na\ne P\na$.  Now suppose that $P\na$ is branched and that $J$ is a $P\na$-primary ideal of $\na$ distinct from $P\na$.  Then it is straightforward to show that $J\cap D$ is distinct from $P$ and is $P$-primary.


(3) Let $\psi$ be a nonzero element of $\na$, and let $N$ be a maximal ideal with $\psi \in N$.  Then $\psi \na=f\na$ for some $f \in D[X]$, and writing $N=M\na$ for some $M \in \QMax^{\stf}(D)$ (Lemma~\ref{l:nagata}(3)), we must have $f \in MD[X]$ and hence $c(f) \subseteq M$.  Therefore, if $D$ has finite $\stf$-character, then $\na$ has finite character. The converse is even more straightforward.


(4) Suppose that $I$ is a radical ideal of $D$, and let $\psi^n \in I{\rm Na}(D,\st)$ for some $\psi \in {\rm Na}(D,\st)$ and positive integer $n$.  Then there is an element $g \in N(\st)$ with ($g\psi^n$ and hence) $(g\psi)^n \in ID[X]$.  Since $ID[X]$ is a radical ideal of $D[X]$, $g\psi \in ID[X]$ and we must have  $\psi \in I{\rm Na}(D,\st)$.  Therefore, $I{\rm Na}(D,\st)$ is a radical ideal of ${\rm Na}(D,\st)$.
The converse follows easily from  the fact that $I\Na(D,\star) \cap D = I^{\stt} \cap D=I$ (Lemma~\ref{l:nagata}(7)).

(5) (a) This is part of \cite[Proposition 3.5]{fhp}, but we give here a proof more in the spirit of this paper. Let $P\in \QSpec^{\stf}(D)$.  If $P$ is $\stf$-sharp, then by \cite[Proposition 3.1]{fhp} $P$ contains a finitely generated ideal $I$ with $I \nsubseteq M$ for all $M \in \nabla^{\stf}(P)$, and, using the description of $\Max(\na)$ given in Lemma~\ref{l:nagata}(3), $I\na$ is a finitely generated ideal of $\na$ contained in $P\na$ but in no element of $\nabla(P\na)$. Therefore, $P\na$ is sharp in the Pr\"{u}fer domain $\na$.
 For the converse, assume that $P\na$ is sharp in $\na$. Then $P\na$ contains a finitely generated ideal $J$ with $J \subseteq P\na$ but $J \nsubseteq N$ for $N \in \nabla(P\na)$ \cite[Corollary 2]{GH-67}.  Then $J=I\na$ for some finitely generated ideal $I$ of $D$ (necessarily) contained in $P$ by Lemma~\ref{l:p*md}(2c), and it is easy to see that $I \nsubseteq M$ for $M \in \nabla^{\stf}(D)$. Then by \cite[Proposition 3.1]{fhp}, $P$ is $\stf$-sharp.
 Statement (b) follows easily from (a) (using Lemma~\ref{l:p*md}).
\end{proof}

Let $D$ be an almost Dedekind domain with a non-finitely generated maximal ideal $M$.  Then $M^{-1}=D$, but $M$ is not idempotent (since $MD_M$ is not idempotent in the Noetherian valuation domain $D_M$). Our next result shows that this cannot happen in a sharp Pr\"ufer doman.

\begin{theorem} \label{t:sharp} Let $D$ be a Pr\"ufer domain.  If $D$ is $(d$-$)$sharp and  $I$ is a nonzero ideal of $D$ with $I^{-1}=D$, then $I$ is idempotent.
\end{theorem}
\begin{proof} Assume that $D$ is sharp.  Proceeding contrapositively, suppose that $I$ is a nonzero, non-idempotent ideal of $D$.   Then, for some maximal ideal $M$ of $D$, $ID_M$ is not idempotent in $D_M$.  Since $D$ is a sharp domain, we may choose a finitely generated ideal $J$ of $D$ with $J \subseteq M$
 but $J \nsubseteq N$ for all maximal ideals $N \ne M$.  Since $ID_M$ is a non-idempotent ideal in the valuation domain $D_M$, there is an element $a \in I$ for which
 $I^2D_M \subsetneq aD_M$.  Let $B:=J+Da$.  Then $I^2D_M \subseteq BD_M$ and,
  for $N \in \Max(D) \setminus \{M\}$, $I^2D_N \subseteq D_N = BD_N$.  Hence $I^2 \subseteq B$.
  Since $B$ is a proper finitely generated ideal, we then have $(I^2)^{-1} \supseteq B^{-1} \supsetneq D$.  Hence $(I^2)^{-1} \ne D$, from which it follows that $I^{-1}\ne D$, as desired.
\end{proof}

Kang \cite[Proposition 2.2]{K-89} proves that, for a nonzero ideal $I$ of $D$, we always have $I^{-1}{\rm Na}(D,v)=({\rm Na}(D,v)):I)$. This cannot be extended to general semistar Nagata rings; for example, if $D$ is an almost Dedekind domain with non-invertible maximal ideal $M$ and we define a semistar operation $\st$ by $E^{\st}=ED_M$ for $E \in \Fb(D)$, then $(D:M)=D$ and hence $(D:M)\na=\na =D[X]_{M[X]}=D_M(X) \linebreak \subsetneq (D_M:MD_M)D_M(X) =(\na:M\na)$ (where the proper inclusion holds because $MD_M$ is principal in $D_M$).  At any rate, what we really require is the equality $(D^{\st}:E)\na=(\na:E)$ for $E \in \boldsymbol{\overline{F}}(D)$.  In the next lemma, we show that this holds in a P$\st$MD but not in general. The proof of part (1) of the next lemma is a relatively straightforward translation of the proof of \cite[Proposition 2.2]{K-89} to the semistar setting.  In carrying this out, however, we discovered a minor flaw in the proof of \cite[Proposition 2.2]{K-89}. The flaw involves a reference to \cite[Proposition 34.8]{G}, but this result requires that the domain $D$ be integrally closed. To ensure complete transparency, we give the proof in full detail.

\begin{lemma} \label{l:ext} Let $\st$ be a semistar operation on $D$. Then: \begin{enumerate}
\item $(D^{\st}:E){\rm Na}(D,\st)\supseteq ({\rm Na}(D,\st):E)$ for each $E \in \boldsymbol{\overline{F}}(D)$.
\item The following statements are equivalent: \begin{enumerate}
\item $(D^{\st}:E){\rm Na}(D,\st)=({\rm Na}(D,\st):E)$ for each $E \in \boldsymbol{\overline{F}}(D)$.
\item $D^{\st}=D^{\stt}$.
\item $D^{\st} \subseteq {\rm Na}(D,\st)$.
\end{enumerate}
\item $(D^{\stt}:E)\na=(\na:E)$ for each $E \in \boldsymbol{\overline{F}}(D)$.
\item If $D$ is a P$\st$MD, then the equivalent conditions in (2) hold. 
\end{enumerate}
\end{lemma}
\begin{proof} (1) Let $E \in \boldsymbol{\overline{F}}(D)$, and let $\psi \in ({\rm Na}(D,\st):E)$.  For $a \in E$, $a \ne 0$, we may find $g \in N(\st)$ such that $\psi ag \in D[X]$.  This yields $\psi g \in a^{-1}D[X] \subseteq K[X]$, and hence $\psi = f/g$ for some $f \in K[X]$.  We claim that $c(f) \subseteq (D^{\st}:E)$.  Granting this, we have $f \in (D^{\st}:E)D[X]$, from which it follows that $\psi = f/g \in (D^{\st}:E){\rm Na}(D,\st)$, as desired.  To prove the claim, take $b \in E$, and note that $fb\in {\rm Na}(D,\st)$.  Hence $fbh \in D[X]$ for some $h \in N(\st)$, and so $c(fh)b\subseteq D$.  By the content formula \cite[Theorem 28.1]{G}, there is an integer $m$ for which $c(f)c(h)^{m+1}=c(fh)c(h)^m$.  Using the fact that $c(h)^{\st}=D^{\st}$, we obtain $c(f)^{\st}=c(fh)^{\st}$ and hence that $c(f)b \subseteq c(fh)^{\st}b \subseteq D^{\st}$.  Therefore, $c(f) \subseteq (D^{\st}:E)$, as claimed.

(2) Under the assumption in (c), $D^{\st} \subseteq {\rm Na}(D,\st) \cap K=D^{\stt}$ (Lemma~\ref{l:nagata}(6)).  Hence (c) $\Rightarrow$ (b). Now assume that $D^{\st}=D^{\stt}$.  Then for $E \in \boldsymbol{\overline{F}}(D)$, we have $(D^{\st}:E)E \subseteq D^{\st}=D^{\stt} \subseteq {\rm Na}(D,\st)$; using (1), the implication (b) $\Rightarrow$ (a) follows.  That (a) $\Rightarrow$ (c) follows upon taking $E=D$ in (a).

(3) This follows easily from (2), because $\na={\rm Na}(D, \stt)$ by Lemma~\ref{l:nagata}(5).
%
%
%
%

(4) This follows from (2), since if $D$ is a P$\st$MD, then $D^{\st}=D^{\stt}$ by Lemma~\ref{l:p*md}(2a).
\end{proof}


The conditions in Lemma~\ref{l:ext}(2) need not hold: Let $F \subsetneq k$ be fields, $V=k[[x]]$ the power series ring over $V$ in one variable, and $D=F+M$, where $M=xk[[x]]$. Define a (finite-type)  semistar operation $\st$ on $D$ by $A^{\st}=AV$ for $A \in \FF(D)$. Then $D^{\st}=V \supsetneq D=D_M=D^{\stt}$.

We can now extend Theorem~\ref{t:sharp} to P$\star$MDs.

\begin{corollary} \label{c:sharp} 
Let $\st$ be a semistar operation on $D$ such that $D$ is a $\stf$-sharp P$\st$MD, and let $I$ be a nonzero ideal of $D$ with $(D^{\st}:I)=D^{\st}$. Then $I$ is $\stf$-idempotent.
\end{corollary}
\begin{proof} By Lemma~\ref{l:ext}(3), we have $$({\rm Na}(D,\st):I{\rm Na}(D,\st))=(D^{\st}:I){\rm Na}(D,\st)=D^{\st}{\rm Na}(D,\st)={\rm Na}(D,\st).$$ Hence $I{\rm Na}(D,\st)$ is idempotent in the Pr\"ufer domain ${\rm Na}(D,\st)$ by Theorem~\ref{t:sharp}.  Lemma~\ref{l:branched-idempotent}(1) then yields that $I$ is $\stf$-idempotent.
\end{proof}

Many semistar counterparts of ideal-theoretic properties in domains result in equations that are  ``external'' to $D$, since for a semistar operation $\st$ on $D$ and a nonzero ideal $I$ of $D$, it is possible that $I^{\st} \nsubseteq D$. Of course, $\st$-idempotence is one such property.  Often, one can obtain a ``cleaner'' counterpart by specializing from P$\st$MDs to ``ordinary'' P$v$MDs.  We recall some terminology.  Semistar operations are generalizations of \emph{star} operations, first considered by Krull and repopularized by Gilmer \cite[Sections 32, 34]{G}. Roughly, a star operation is a semistar operation restricted to the set $\F(D)$ of nonzero fractional ideals of $D$ with the added requirement that one has $D^{\st}=D$. The most important star operation (aside from the $d$-, or trivial, star operation) is the \emph{$v$-operation}: For $E \in \F(D)$, put $E^{-1}=\{x \in K \mid xE \subseteq D\}$ and $E^v=(E^{-1})^{-1}$.  Then $v_{\! f}$ (restricted to $\F(D)$) is the \emph{$t$-operation} and $\widetilde v$ is the $w$-operation. Thus a \emph{P$v$MD} is a domain in which each nonzero finitely generated ideal is $t$-invertible.  Corollary~\ref{c:sharp} then has the following restricted interpretation (which has the advantage of being \emph{internal} to $D$).

\begin{corollary} \label{c:converse-pvmd} If $D$ is a $t$-sharp P$v$MD and $I$ is a nonzero ideal of $D$ for which $I^{-1}=D$, then $I$ is $t$-idempotent. 
\end{corollary}

Our next result is a partial converse to Theorem~\ref{t:sharp}.

\begin{proposition} \label{p:converse}
 Let $D$ be a Pr\"ufer domain such that $I$ is idempotent whenever $I$ is a nonzero ideal of $D$ with $I^{-1} =D$.  Then, every branched maximal ideal of $D$ is sharp.
\end{proposition}
\begin{proof} Let $M$ be a branched maximal ideal of $D$.
Then $MD_M =\mbox{rad}({aD_M})$ for some nonzero element $a \in M$   \cite[Theorem 17.3]{G}.
Let $I:=aD_M \cap D$.  Then $I$ is $M$-primary, and since $ID_M=aD_M$, ($ID_M$ and hence) $I$ is not idempotent.
By hypothesis, we may choose $u \in I^{-1} \setminus D$.
Since $Iu \subseteq D$ and $ID_N =D_N$ for $N \in \Max(D) \setminus \{M\}$,  then $u \in \bigcap \{D_N \mid N \in \textrm{Max}(D), N \ne M\}$.  On the other hand, since $u \notin D$, $u \not\in D_M$.  It follows that $M$ is sharp.
\end{proof}

Now we extend Proposition~\ref{p:converse} to P$\st$MDs.

\begin{corollary} \label{c:converse}
Let $\st$ be a semistar operation on $D$, and assume that $D$ is a P$\star$MD such that $I$ is $\stf$-idempotent whenever $I$ is a nonzero 
ideal of $D$ with $(D^{\star}:I)=D^{\star}$. 
Then, each branched quasi-$\stf$-maximal ideal of $D$ is $\stf$-sharp. $($In particular if $D$ is a P$v$MD in which $I$ is $t$-idempotent whenever $I$ is a nonzero ideal of $D$ with $I^{-1}=D$, then each branched maximal $t$-ideal of $D$ is $t$-sharp.$)$
\end{corollary}
\begin{proof}
Let $J$ be a a nonzero ideal  of the Pr\"ufer domain $\Na(D,\star)$ with $(\Na(D,\star):J)=\Na(D,\star)$.
By Lemma~\ref{l:p*md}(1c), $J = I\Na(D,\star)$ for some ideal $I$ of $D$.
Applying Lemma~\ref{l:ext}(3) and Lemma~\ref{l:nagata}(6), we obtain $(D^{\st}:I)=D^{\st}$. Hence, by hypothesis, $I$ is $\stf$-idempotent, and this yields that
$J=I\Na(D, \star)$ is idempotent in the Pr\"ufer domain $\Na(D,\star)$ (Lemma~\ref{l:branched-idempotent}(1)).
Now, let $M$ be a branched quasi-$\stf$-maximal ideal of $D$. Then, by Lemma~\ref{l:branched-idempotent}(2), $M\Na(D,\star)$ is a branched maximal ideal of $\Na(D,\star)$.  We may now apply Proposition~\ref{p:converse} to conclude that $M\Na(D,\star)$ is sharp.  Therefore, $M$ is $\stf$-sharp in $D$ by 
Lemma~\ref{l:branched-idempotent}(5).
\end{proof}

If $P$ is a prime ideal of a Pr\"ufer domain $D$, then powers of $P$ are $P$-primary by \cite[Theorem 23.3(b)]{G}; it follows that $P$ is idempotent if and only if $PD_P$ is idempotent. We use this fact in the next result.

 It is well known that a proper idempotent ideal of a valuation domain must be prime \cite[Theorem 17.1(3)]{G}.  In fact, according to \cite[Exercise 3, p. 284]{G}, a proper idempotent ideal in a Pr\"ufer domain must be a radical ideal.  We (re-)prove  and extend this fact and add a converse:

\begin{theorem} \label{t:idemp} Let $D$ be a Pr\"ufer domain, and let $I$ be an ideal of $D$.  Then $I$ is idempotent if and only if $I$ is a radical ideal each of whose minimal primes is idempotent.
\end{theorem}
\begin{proof} The result is trivial for $I=(0)$ and vacuously true for $I=D$. Suppose that $I$ is a proper nonzero idempotent ideal of $D$, and let $P$ be a prime minimal over $I$.  Then $ID_P$ is idempotent, and we must have $ID_P=PD_P$ \cite[Theorem 17.1(3)]{G}. Hence $PD_P$ is idempotent, and therefore, by the comment above, so is $P$.   Now let $M$ be a maximal ideal containing $I$.  Then $ID_M$ is idempotent, hence prime (hence radical).  It follows (checking locally) that $I$ is a radical ideal.

Conversely, let $I$ be a radical ideal each of whose minimal primes is idempotent.  If $M$ is a maximal ideal containing $I$ and $P$ is a minimal prime of $I$ contained in $M$, then $ID_M=PD_M$.  Since $P$ is idempotent, this yields $ID_M=I^2D_M$.  It follows that $I$ is idempotent.
\end{proof}

We next extend Theorem~\ref{t:idemp} to P$\star$MDs.

\begin{corollary} \label{c:idemp} Let $D$ be a P$\star$MD, where $\star$ is a semistar operation on $D$, and let $I$ be a quasi-$\stf$-ideal of $D$.  Then $I$ is $\stf$-idempotent if and only if $I$ is a radical ideal each of whose minimal primes is $\stf$-idempotent. $($In  particular,  if $D$ is a P$v$MD and $I$ is a $t$-ideal of $D$, then $I$ is $t$-idempotent if and only if $I$ is a radical ideal each of whose minimal primes is $t$-idempotent.$)$
\end{corollary}
\begin{proof}  Suppose that $I$ is $\stf$-idempotent. Then $I\Na(D,\star)$ is an idempotent ideal in $\Na(D,\star)$  by Lemma~\ref{l:branched-idempotent}(1).
By Theorem~\ref{t:idemp}, $I\Na(D,\star)$ is a radical ideal of $\Na(D,\star)$, and hence, by Lemma~\ref{l:branched-idempotent}(4), $I$ is a radical ideal of $D$. Now let $P$ be a minimal prime of $I$ in $D$. Then $P$ is a quasi-$\stf$-prime of $D$.  By Lemma~\ref{l:p*md}(2b) $P\Na(D,\star)$ is minimal over $I\Na(D,\star)$, whence $P\Na(D,\star)$ is idempotent, again by Theorem~\ref{t:idemp}. The $\stf$-idempotence of $P$ now follows from Lemma~\ref{l:branched-idempotent}(1).

The converse follows by similar applications of Theorem~\ref{t:idemp} and Lemma~\ref{l:branched-idempotent}.
\end{proof}

Recall that a Pr\"ufer domain is said to be \emph{strongly discrete} (\emph{discrete}) if it has no nonzero  (branched) idempotent prime ideals. Since unbranched primes in a Pr\"ufer domain must be idempotent \cite[Theorem 23.3(b)]{G}, a  Pr\"ufer domain is strongly discrete if and only if it is discrete and has no unbranched prime ideals.  We have the following straightforward application of Theorem~\ref{t:idemp}.


\begin{corollary} \label{c:strdisc} Let $D$ be a Pr\"ufer domain. \begin{enumerate}
\item If $D$ is discrete, then an ideal $I$ of $D$ is idempotent if and only if $I$ is a radical ideal each of whose minimal primes is unbranched.
\item If $D$ is strongly discrete, then $D$ has no proper nonzero idempotent ideals.
\end{enumerate}
\end{corollary}

 Let us call a P$\st$MD \emph{$\stf$-strongly discrete} (\emph{$\stf$-discrete}) if it has no (branched) $\stf$-idempotent quasi-$\stf$-prime ideals.  From Lemma~\ref{l:branched-idempotent}(1,2), we have the usual connection between a property of a P$\st$MD and the corresponding property of its $\st$-Nagata ring:

\begin{proposition} \label{p:discrete} Let $\st$ be a semistar operation on $D$.  Then $D$ is $\stf$-$($strongly$)$ discrete P$\st$MD if and only if ${\rm Na}(D,\st)$ is a $($strongly$)$ discrete Pr\"ufer domain. 
\end{proposition}

Applying Corollary~\ref{c:idemp} and Lemma~\ref{l:branched-idempotent}(1,2), we have the following extension of Corollary~\ref{c:strdisc}.

\begin{corollary} \label{c:starstrdisc} \noindent  Let $D$ be a domain. \begin{enumerate}
\item Assume that $D$ is a P$\st$MD for some semistar operation $\st$ on $D$. \begin{enumerate}
\item If $D$ is $\stf$-discrete, then a nonzero quasi-$\stf$-ideal $I$ of $D$ is $\stf$-idempotent if and only if $I$ is a radical ideal each of whose minimal primes is unbranched.
\item If $D$ is $\stf$-strongly discrete, then $D$ has no $\stf$-proper $\stf$-idempotent  ideals.
\end{enumerate}
\item Assume that $D$ is a P$v$MD. \begin{enumerate}
\item If $D$ is $t$-discrete, then a $t$-ideal $I$ of $D$ is $t$-idempotent if and only if $I$ is a radical ideal each of whose minimal primes is unbranched.
\item If $D$ is $t$-strongly discrete, then $D$ has no $t$-proper $t$-idempotent  ideals.
\end{enumerate}
\end{enumerate}
\end{corollary}

\section{Divisoriality}


According to \cite
[Corollary 4.1.14]{FHP97}, if $D$ is a doublesharp Pr\"ufer domain and $P$ is a nonzero, nonmaximal ideal of $D$, then $P$ is divisorial.  The natural question arises: If $D$ is a $\stf$-doublesharp P$\st$MD and $P \in \QSpec^{\stf}(D)\setminus \QMax^{\stf}(D)$, is $P$ necessarily divisorial?  Since $\st$ is an arbitrary semistar operation and divisoriality specifically involves the $v$-operation, one might expect the answer to be negative.  Indeed, we give a counterexample in Example~\ref{e:div} below.  However, in Theorem~\ref{t:##} we  prove a general result, a corollary of which does yield divisoriality in the ``ordinary'' P$v$MD case.  First, we need a lemma, the first part of which may be regarded as an extension of \cite[Proposition 2.2(2)]{K-89}.

\begin{lemma} \label{l:div} Let $\st$ be a semistar operation on $D$.  Then \begin{enumerate}
\item $(D^{\stt}:(D^{\stt}:E))\na=(\na:(\na:E))$ for each $E \in \Fb(D)$, and 
\item if $I$ is a nonzero ideal of $D$, then $I^{\stt}$ is a divisorial ideal of $D^{\stt}$ if and only if $I\na$ is a divisorial ideal of $\na$.
\end{enumerate}
In particular, if $D$ is a P$\st$MD, then $(D^{\st}:(D^{\st}:E))\na=(\na:(\na:E))$ for each $E \in \Fb(D)$; and, for a nonzero ideal $I$ of $D$, $I^{\stf}$ is divisorial in $D^{\st}$ if and only if $I\na$ is divisorial in $\na$.
\end{lemma}
\begin{proof} Set $\cal N=\na$.  For (1), applying Lemma~\ref{l:ext}, we have $$(D^{\stt}:(D^{\stt}:E))\cal N=(\cal N:(D^{\stt}:E))=(\cal N:(\cal N:E)).$$

(2)  Assume that $I$ is a nonzero ideal of $D$.  If $I^{\stt}$ is divisorial in $D^{\stt}$, then (using (1)) $$(\cal N:(\cal N:I\cal N))=(D^{\stt}:(D^{\stt}:I^{\stt}))\cal N=I^{\stt}\cal N=I\cal N.$$

Now suppose that $I\cal N$ is divisorial.  Then $$(D^{\stt}:(D^{\stt}:I^{\stt}))\cal N=(\cal N:(\cal N:I))=I\cal N,$$ whence $$(D^{\stt}:(D^{\stt}:I^{\stt}))\subseteq I\cal N \cap K=I^{\stt}.$$
The ``in particular'' statement follows from standard considerations.
\end{proof}

\begin{theorem} \label{t:##} Let $\st$ be a semistar operation on $D$ such that $D$ is a $\stf$-doublesharp P$\st$MD, and let $P \in \QSpec^{\stf}(D) \setminus \QMax^{\stf}(D)$.  Then $P^{\stf}$ is a divisorial ideal of $D^{\st}$.
\end{theorem}
\begin{proof} Since ${\rm Na}(D,\st)$ is a doublesharp Pr\"ufer domain (Lemma~\ref{l:branched-idempotent}(5)), $P{\rm Na}(D,\st)$ is divisorial by \cite[Corollary 4.1.14]{FHP97}.  Hence $P^{\stf}$ is divisorial in $D^{\st}$ by Lemma~\ref{l:div}.
\end{proof}

\begin{corollary} \label{c:pvmd} If $D$ is a $t$-doublesharp P$v$MD,  and $P$ is a non-$t$-maximal $t$-prime of $D$, then 
$P$ is divisorial.
\end{corollary}
\begin{proof}
Take $\st=v$ in Theorem~\ref{t:##}. (More precisely, take $\st$ to be any extension of the star operation $v$ on $D$ to a semistar operation on $D$, so that $\stf$ (restricted to $D$) is the $t$-operation on $D$.) Then $P=P^t=P^{\stf}$ is divisorial by Theorem~\ref{t:##}.
\end{proof}

\begin{example} \label{e:div} 
Let $p$ be a prime integer and let $D:={\rm Int}(\mathbb Z_{(p)})$.
Then $D$ is a 2-dimensional Pr\"{u}fer domain by \cite[Lemma VI.1.4 and Proposition V.1.8]{cc}.
Choose a height 2 maximal ideal $M$ of $D$, and let $P$ be a height 1 prime ideal of $D$ contained in $M$.
Then $P=q\mathbb Q[X]\cap D$ for some irreducible polynomial $q\in \mathbb Q[X]$ \cite[Proposition V.2.3]{cc}.
By \cite[Theorems VIII.5.3 and VIII.5.15]{cc}, $P$ is not a divisorial ideal of $D$. Set $E^{\st}=ED_M$ for $E \in \Fb(D)$.
Then, $\star$ is a finite-type semistar operation on $D$. Clearly, $M$ is the only quasi-$\st$-maximal ideal of $D$, and, since $D_M$ is a valuation domain, $D$ is a P$\st$MD by Lemma~\ref{l:p*md}.  Moreover, ${\rm Na}(D,\st)=D_M(X)$ is also a valuation domain and hence a doublesharp Pr\"ufer domain, which yields that $D$ is a $\stf$-doublesharp P$\st$MD $($Lemma~\ref{l:branched-idempotent}$)$. Finally, since $P=PD_M \cap D=P^{\st}\cap D$, $P$ is a non-$\stf$-maximal quasi-$\stf$-prime of $D$. \qed
 \end{example}


In the remainder of the paper, we impose on Pr\"ufer domains (P$\st$MDs) the finite character (finite $\stf$-character) condition.  As we shall see, this allows improved versions of Theorem~\ref{t:idemp} and Corollary~\ref{c:idemp}.  It also allows a type of unique factorization for  (quasi-$\stf$-)ideals that are simultaneously ($\stf$-)idempotent and ($\stf$-)divisorial.

\begin{theorem} \label{t:idemp-div} Let $D$ be a Pr\"ufer domain with finite character, and let $I$ be a nonzero ideal of $D$.  Then:\begin{enumerate}
\item $I$ is idempotent if and only if $I$ is a product of idempotent prime ideals.
\item The following statements are equivalent. \begin{enumerate}
\item $I$ is idempotent and divisorial.
\item $I$ is a product of non-maximal idempotent prime ideals.
\item $I$ is a product of divisorial idempotent prime ideals.
\item $I$ has a unique representation as the product of incomparable divisorial idempotent primes.
\end{enumerate}
\end{enumerate}
\end{theorem}
\begin{proof} (1) Suppose that $I$ is idempotent. By Theorem~\ref{t:idemp}, $I$ is the intersection of its minimal primes, each of which is idempotent.  Since $D$ has finite character, $I$ is contained in only finitely many maximal ideals, and, since no two distinct minimal primes of $I$ can be contained in a single maximal ideal, $I$ has only finitely many minimal primes and they are comaximal.  Hence $I$ is the product of its minimal primes (and each is idempotent).  The converse is trivial.

(2) (a) $\Rightarrow$ (b): Assume that $I$ is idempotent and divisorial.  By (1) and its proof,   $I=P_1 \cdots P_n=P_1 \cap \cdots \cap P_n$, where the $P_i$ are the minimal primes of $I$.  We claim that each $P_i$ is divisorial.  To see this, observe that $$(P_1)^vP_2 \cdots P_n \subseteq (P_1 \cdots P_n)^v=I^v=I\subseteq P_1.$$  Since the $P_i$ are incomparable, this gives $(P_1)^v \subseteq P_1$, that is, $P_1$ is divisorial. By symmetry each $P_i$ is divisorial.  It is well known that in a Pr\"ufer domain, a maximal ideal cannot be both idempotent and divisorial. Hence the $P_i$ are non-maximal.

(b) $\Rightarrow$ (c): Since $D$ has finite character, it is a ($d$)-doublesharp Pr\"ufer domain \cite[Theorem 5]{GH-67}, whence nonmaximal primes are automatically divisorial by \cite[Corollary 4.1.14]{FHP97}.

(c) $\Rightarrow$ (a): Write $I=Q_1 \cdots Q_m$, where each $Q_j$ is a divisorial idempotent prime.  Since $I$ is idempotent (by (1)), we may also write $I=P_1 \cdots P_n$, where the $P_i$ are the minimal primes of $I$. For each $i$, we have $Q_1 \cdots Q_m= I\subseteq P_i$, from which it follows that $Q_j \subseteq P_i$ for some $j$.  By minimality, we must then have $Q_j=P_i$.  Thus each $P_i$ is divisorial, whence $I=P_1 \cap \cdots \cap P_n$ is divisorial.

Finally, we show that (d) follows from the other statements.  We use the notation in the proof of  (c) $\Rightarrow$ (a). In the expression $I=P_1 \cdots P_n$, the $P_i$ are (divisorial, idempotent, and) incomparable, and it is clear  that no $P_i$ can be omitted. To see that this is the only such expression, consider a representation $I=Q_1\cdots Q_m$, where the $Q_i$ are divisorial, idempotent, and incomparable.  
Fix a $Q_k$.  Then $P_1\cdots P_n =I \subseteq Q_k$, and we have $P_i\subseteq Q_k$ for some $i$.  However, as above, $Q_j \subseteq P_i$ for some $j$, whence, by incomparability, $Q_k=P_i$. The conclusion now follows easily.
\end{proof}

We note that incomparability is necessary for uniqueness above--for example, if $D$ is a valuation domain and $P \subsetneq Q$ are non-maximal (necessarily divisorial) primes, then $P=PQ$.

We close by extending Theorem~\ref{t:idemp-div} to P$\st$MDs and then to ``ordinary'' P$v$MDs.  We omit the (by now) straightforward proofs.


\begin{corollary} \label{c:idemp-div} Let $\st$ be a semistar operation on $D$ such that $D$ is  a P$\st$MD with finite $\stf$-character, and let $I$ be a quasi-$\stf$-ideal of $D$.  Then:\begin{enumerate}
\item $I$ is $\stf$-idempotent if and only if $I^{\stf}$ is a $\stf$-product of $\stf$-idempotent quasi-$\stf$-prime ideals in $D$, that is, $I^{\stf}=(P_1\cdots P_n)^{\stf}$, where the $P_i$ are $\stf$-idempotent quasi-$\stf$-primes of $D$.
\item The following statements are equivalent. \begin{enumerate}
\item $I$ is $\stf$-idempotent and $\stf$-divisorial $(I^{\stf}$ is divisorial in $D^{\st})$.
\item $I$ is a $\stf$-product of non-quasi-$\stf$-maximal idempotent quasi-$\stf$-prime ideals.
\item $I$ is a $\stf$-product of $\stf$-divisorial $\stf$-idempotent prime ideals.
\item $I$ has a unique representation as a $\stf$-product of incomparable $\stf$-divisorial $\stf$-idempotent primes.
\end{enumerate}
\end{enumerate}
\end{corollary} 
%
%

\begin{corollary} \label{c:idemp-div-pvmd} Let $D$ be a P$v$MD with finite $t$-character, and let $I$ be a nonzero $t$-ideal of $D$.  Then:\begin{enumerate}
\item $I$ is $t$-idempotent if and only if $I$ is a $t$-product of $t$-idempotent $t$-prime ideals in $D$, 
\item The following statements are equivalent. \begin{enumerate}
\item $I$ is $t$-idempotent and divisorial.
\item $I$ is a $t$-product of non-$t$-maximal $t$-idempotent $t$-primes.
\item $I$ is a $t$-product of divisorial $t$-idempotent $t$-primes.
\item $I$ has a unique representation as a $t$-product of incomparable divisorial $t$-idempotent $t$-primes.
\end{enumerate}
\end{enumerate}
\end{corollary}

 \end{document}